\documentclass[a4paper,11pt,reqno]{amsart}
\usepackage{amsmath}
\usepackage{amssymb}

\usepackage{mathrsfs}
\renewcommand{\mathcal}{\mathscr}

\theoremstyle{plain} 
\newtheorem{theorem}{Theorem}[section]

\newtheorem{lemma}[theorem]{Lemma}

\newtheorem{remark}[theorem]{Remark}
\newtheorem{example}[theorem]{Example}

\newtheorem*{Otheorem}{Theorem}

\makeatletter
    
    \@addtoreset{equation}{section}
  \makeatother

\title[Takagi class functions]{On the rate of convergence for Takagi class functions}
\author{Shoto Osaka}
\address{Graduate School of Engineering Science, Yokohama National University}
\email{osaka-shoto-jc@ynu.jp}
\author{Masato Takei}
\address{Department of Applied Mathematics, Faculty of Engineering, Yokohama National University}
\email{takei-masato-fx@ynu.ac.jp}

\begin{document}
\begin{abstract}
We consider a generalized version of the Takagi function,
which is one of the most famous example of nowhere differentiable continuous functions.
We investigate a set of conditions to describe the rate of convergence of Takagi class functions from the probabilistic point of view: The law of large numbers, the central limit theorem, and the law of iterated logarithm. On the other hand, we show that the Takagi function itself does not satisfy the law of large numbers in the usual sense.
\end{abstract}

\maketitle

\section{Introduction}
\label{intro}

The tent-map on $[0,1]$ is defined by
\[ \varphi(x):= 
\begin{cases} 
 2x &(x \in [0,1/2]), \\ 
 2(1-x)&(x \in [1/2,1]). \\
\end{cases}
\]
Let $\varphi^{(1)}(x):=\varphi(x)$, and $\varphi^{(n)}(x)$ denote the $n$-fold iteration of $\varphi(x)$.
For $x \in \mathbb{R}$, we write $\varphi(x)$ for $\varphi(x-\lfloor x \rfloor)$.
Then we can see that $\varphi(x)$ has period 1, and
\[ \varphi^{(n)}(x) = \varphi(2^{n-1} x) \quad \mbox{for each $n=1,2,\cdots$.} \]

The continuous function defined by
\[
T(x):=\sum_{n=1}^{\infty} \dfrac{1}{2^n} \varphi^{(n)}(x)
\]
is called the {\it Takagi function}.
Takagi \cite{Takagi1903} shows that $T(x)$ has nowhere finite derivative (the above definition is different from but equivalent to the original one given in \cite{Takagi1903}). We refer to an excellent survey paper by Allaart and Kawamura \cite{AllaartKawamura11/12} for several known properties of $T(x)$ and its generalizations.

We say a real sequence $\{ c_n \} \in \ell^p$ if 
\[
\sum_{n=1}^{\infty} |c_n|^p < +\infty.
\]
If $\{c_n\} \in \ell^1$, then for each $x \in [0,1]$, the limit
\begin{equation}
f(x):=\sum_{n=1}^{\infty} c_n \varphi^{(n)}(x) \label{eq:TakagiClassFunctionDef}
\end{equation}
exists, and in fact the convergence is uniform over $[0,1]$: If we set
\begin{equation}
f_N(x) := \sum_{n=1}^{N-1} c_n \varphi^{(n)}(x) \quad \mbox{for $N=1,2,\cdots$},\label{eq:TakagiClassFunctionDefN}
\end{equation}
then
\begin{align} \label{ineq:TakagiRateConvGen}
\sup_{x \in [0,1]} |f(x)-f_N(x)| \leq \sum_{n=N}^{\infty} |c_n|.
\end{align}
Thus $f(x)$ is a continuous function on $[0,1]$. 
On the other hand, if $f(x)$ converges for all $x \in [0,1]$, then $\{c_n\} \in \ell^1$ (Hata and Yamaguti \cite{HataYamaguti84JJAM}). 
The set of all functions defined by \eqref{eq:TakagiClassFunctionDef} with $\{c_n\} \in \ell^1$ is 
called the {\it Takagi class}. 

K\^{o}no \cite{Kono87ActaMathHungar} studied the differentiability and the modulus of continuity of Takagi class functions, from the probabilistic point of view: We regard the functions $f(x)$ and $f_N(x)$ defined by \eqref{eq:TakagiClassFunctionDef} and \eqref{eq:TakagiClassFunctionDefN} as a random variable on the Lebesgue probability space $(\Omega,\mathcal{F},P)$, where $\Omega=[0,1]$, $\mathcal{F}$ is the Borel $\sigma$-field of $\Omega$, and $P$ is the Lebesgue measure on $\Omega$. We quote the result on the differentiability proved in \cite{Kono87ActaMathHungar}.

\begin{Otheorem}[\cite{Kono87ActaMathHungar}, Theorem 2] Assume that $\{c_n\} \in \ell^1$, and consider the continuous function $f(x)$ defined by \eqref{eq:TakagiClassFunctionDef}.
\begin{itemize}
\item[(i)] If $\{2^n c_n \} \in \ell^2$, then $f(x)$ is absolutely continuous --- differentiable at almost every $x$.
\item[(ii)] If $\{2^n c_n \} \notin \ell^2$ but $\displaystyle \lim_{n \to \infty} 2^n c_n =0$, then $f(x)$ is non differentiable at almost every $x$, but differentiable on an uncountable set. 
\item[(iii)] If $\displaystyle \limsup_{n \to \infty} 2^n |c_n| >0$, then $f(x)$ is nowhere differentiable.
\end{itemize}
\end{Otheorem}

For results on the modulus of continuity, see \cite{Allaart09JMSJ,Gamkrelidze90,Kono87ActaMathHungar}.

In this paper, we investigate the rate of convergence for Takagi class functions:
What is the magnitude of $f(x)-f_N(x)$? Let us begin with a simple observation.
The tent-map $\varphi(x)$ has two fixed points $x=0,\dfrac{2}{3}$ in $[0,1]$.
\begin{itemize}
\item[$\bullet$] For each dyadic rational $x=\dfrac{k}{2^m}$, we have $\varphi^{(n)}(x)=0$ for all $n>m$, and
$f(x)-f_N(x)=0$ for all $N>m$.
\item[$\bullet$]  Since $\varphi^{(n)}\left(\dfrac{2}{3}\right)=\dfrac{2}{3}$ for all $n$, we have
\[ f\left(\dfrac{2}{3}\right)-f_N\left(\dfrac{2}{3}\right) = \dfrac{2}{3}\sum_{n=N}^{\infty} c_n\quad \mbox{for each $N$}. \]
\end{itemize}
Our main result, summarized in the following theorem, shows that those points are rather `exceptional', and describes the magnitude of $f(x)-f_N(x)$ for `typical' $x$.

\begin{theorem} \label{thm:OT19MainResult} Assume that $\{c_n\} \in \ell^1$, and $\sum_{n=N}^{\infty} (c_n)^2 >0$ for any $N$. In (i) and (ii) below, we assume $\sum_{n=N}^{\infty} c_n \neq 0$ for any $N$ in addition. 
\begin{itemize} 
\item[(i)] (the $L^2$-weak law of large numbers for the ratio) 
\[ \lim_{N \to \infty} \int_0^1 \left(\dfrac{f(x)-f_N(x)}{\tfrac{1}{2} \sum_{n=N}^{\infty} c_n}-1\right)^2 \,dx  = 0 \]
holds if and only if 
\begin{equation}
\lim_{N \to \infty} \dfrac{\sum_{n=N}^{\infty} (c_n)^2}{\left(\sum_{n=N}^{\infty} c_n\right)^2}=0. \label{eq:L2LLNcond}
\end{equation}
Under \eqref{eq:L2LLNcond}, by Chebyshev's inequality,  
\[ \lim_{N \to \infty} P\left( \left\{ x \in [0,1] : 1-\varepsilon  \leq \dfrac{f(x)-f_N(x)}{\tfrac{1}{2}\sum_{n=N}^{\infty} c_n} \leq 1+\varepsilon \right\} \right) = 1\]
for any $\varepsilon>0$.
\item[(ii)] (the strong law of large numbers for the ratio) If 
\begin{equation}
\sum_{N=1}^{\infty} \exp\left\{-K \cdot \dfrac{\left(\sum_{n=N}^{\infty} c_n \right)^2}{\sum_{n=N}^{\infty} (c_n)^2}\right\}
<+\infty  \label{eq:sLLNcond}
\end{equation}
holds for any $K>0$, then
\[ \lim_{N \to \infty} \dfrac{f(x)-f_N(x)}{\tfrac{1}{2} \sum_{n=N}^{\infty} c_n}  = 1 \quad \mbox{for a.e. $x$.} \]
\item[(iii)] (the central limit theorem) If 
\begin{equation}
\lim_{N \to \infty} \dfrac{(c_N)^2}{ \sum_{n=N}^{\infty} (c_n)^2}=0,\label{eq:HallHeyde80Cor3.4-1suff1forULIL}
\end{equation}
then
\begin{align*}
& \lim_{N \to \infty} P\left( \left\{ x \in [0,1] : \dfrac{f(x)-f_N(x)-\tfrac{1}{2} \sum_{n=N}^{\infty} c_n}{\sqrt{\tfrac{1}{12} \sum_{n=N}^{\infty} (c_n)^2}} \leq u \right\} \right) \\
&= \int_{-\infty}^u \dfrac{1}{\sqrt{2\pi}} e^{-t^2/2} \,dt
\end{align*}
for any $u \in \mathbb{R}$. 
\item[(iv)] (the law of iterated logarithm) Let $\phi(t):=\sqrt{2t \log \log (1/t)}$. If \eqref{eq:HallHeyde80Cor3.4-1suff1forULIL} holds, then
\[ \limsup_{N \to \infty} \dfrac{f(x)-f_N(x)-\tfrac{1}{2} \sum_{n=N}^{\infty} c_n}{\phi(\tfrac{1}{12} \sum_{n=N}^{\infty} (c_n)^2)} \leq 1 \quad \mbox{for a.e. $x$.} \]
If \eqref{eq:HallHeyde80Cor3.4-1suff1forULIL} is strengthened to \begin{align}
&\sum_{N=1}^{\infty}\dfrac{(c_N)^4}{\left\{ \sum_{n=N}^{\infty} (c_n)^2 \right\}^2} <+\infty, \label{eq:TakagiClassLIL190310-2}
\end{align}
then we have
\[ \limsup_{N \to \infty} \pm \dfrac{f(x)-f_N(x)-\tfrac{1}{2} \sum_{n=N}^{\infty} c_n}{\phi(\tfrac{1}{12} \sum_{n=N}^{\infty} (c_n)^2)} =1 \quad \mbox{for a.e. $x$.}\]
\end{itemize}
\end{theorem}

\begin{remark} As a corollary of the first half of Theorem \ref{thm:OT19MainResult} (iv), another strong law of large numbers is obtained under a mild condition \eqref{eq:HallHeyde80Cor3.4-1suff1forULIL}: 
\[ \lim_{N \to \infty} \dfrac{f(x)-f_N(x)-\tfrac{1}{2} \sum_{n=N}^{\infty} c_n}{\sum_{n=N}^{\infty} (c_n)^2} =0\quad \mbox{for a.e. $x$.}\] 
\end{remark}

\begin{example} $c_n=n^{-\alpha}$ ($\alpha>1$). We use the fact
\[ \sum_{n=N}^{\infty} n^{-\alpha} \sim \dfrac{1}{\alpha-1} N^{1-\alpha} \quad \mbox{as $N \to \infty$}, \]
where $a_n \sim b_n$ means that $a_n/b_n$ converges to $1$. For $p=1,2,\cdots$, we have
\[  \dfrac{\sum_{n=N}^{\infty} (c_n)^{2p}}{\left(\sum_{n=N}^{\infty} (c_n)^p \right)^2}
 \sim \dfrac{(2p\alpha-1)^{-1} N^{1-2p\alpha}}{(p\alpha-1)^{-2} N^{2(1-p\alpha)}} =\dfrac{(p\alpha-1)^2}{2p\alpha-1} \cdot \dfrac{1}{N} \quad \mbox{as $N \to \infty$}. \]
Since $\sup_{n \geq N} (c_n)^2=(c_N)^2=N^{-2\alpha}$, all the assumptions in Theorem \ref{thm:OT19MainResult} are satisfied.
Noting that 
\begin{align*}
\frac{\sqrt{\tfrac{1}{12} \sum_{n=N}^{\infty} (c_n)^2}}{\tfrac{1}{2} \sum_{n=N}^{\infty} c_n} &\sim \dfrac{\sqrt{\frac{1}{12}\cdot (2\alpha-1)^{-1} N^{1-2\alpha}}}{\frac{1}{2}\cdot (\alpha-1)^{-1}N^{1-\alpha}} \\
&= \dfrac{\alpha-1}{\sqrt{3(2\alpha-1)}} \cdot \dfrac{1}{\sqrt{N}}\quad \mbox{as $N \to \infty$},
\end{align*}
we have the following corollary of Theorem \ref{thm:OT19MainResult} (iii): For any $u \in \mathbb{R}$,
\begin{align*}
 &\lim_{N \to \infty} P\left( \left\{ x \in [0,1] : \left. \left(\frac{f-f_N}{\tfrac{1}{2} \sum_{n=N}^{\infty} c_n}-1\right) \right/ \left(\frac{\alpha-1}{\sqrt{3(2\alpha-1)}} \cdot \frac{1}{\sqrt{N}}\right) \leq u \right\} \right) \\
 &= \int_{-\infty}^u \dfrac{1}{\sqrt{2\pi}} e^{-t^2/2} \,dt.
\end{align*}
\end{example}

\begin{example} $c_n=e^{-K n^{\beta}}$ ($K>0$, $0<\beta<1$). Here we use the following (see  Appendix):
\begin{align}
\sum_{n=N}^{\infty} e^{-K n^{\beta}}\asymp N^{1-\beta} e^{-KN^{\beta}}, \label{ineq:subexpasymp} 
\end{align}
where $a_n \asymp b_n$ means that there are positive constants $C_1$ and $C_2$ such that
\[ C_1 b_n \leq a_n \leq C_2 b_n\quad \mbox{for all $n$.} \]
For $p=1,2,\cdots$, we have
\begin{align*}
\dfrac{\sum_{n=N}^{\infty} (c_n)^{2p}}{\left(\sum_{n=N}^{\infty} (c_n)^p \right)^2} &\asymp \dfrac{N^{1-\beta} e^{-2pKN^{\beta}}}{N^{2(1-\beta)} e^{-2pKN^{\beta}}} = \dfrac{1}{N^{1-\beta} }.
\end{align*}
Noting that
\begin{align*}
\dfrac{\sup_{n \geq N}(c_n)^2}{\sum_{n=N}^{\infty} (c_n)^2} 
\leq \dfrac{e^{-2K N^{\beta}}}{C_1 N^{1-\beta}e^{-2K N^{\beta}}} \to 0 \quad \mbox{as $N \to \infty$},
\end{align*}
we see that the conditions in Theorem \ref{thm:OT19MainResult}, except \eqref{eq:TakagiClassLIL190310-2}, are satisfied. Note that \eqref{eq:TakagiClassLIL190310-2} holds only when $0<\beta<1/2$. At this point we do not know this gap can be filled.
\end{example}

\begin{remark} By the above calculations, we can see the following conditions are sufficient to apply Theorem \ref{thm:OT19MainResult}: $|c_n| \asymp n^{-\alpha}$ with $\alpha>1$, or $|c_n| \asymp e^{-K n^{\beta}}$ with $K>0$ and $0<\beta<1$.
\end{remark}

\begin{example} \label{exs:ExpDame} Suppose that $K>0$ and $\beta \geq 1$. Since $\frac{1}{\beta}-1 \leq 0$, for any $a>0$ satisfying $Ka^{\beta} \geq 1$, we have
\begin{align*}
\int_a^{\infty} e^{-K x^{\beta}}\,dx = \dfrac{1}{\beta K^{\frac{1}{\beta}}}\int_{Ka^{\beta}}^{\infty} t^{\frac{1}{\beta}-1} e^{-t}\,dt 
\leq \dfrac{1}{\beta K^{\frac{1}{\beta}}}\int_{Ka^{\beta}}^{\infty} e^{-t}\,dt 
= \dfrac{1}{\beta K^{\frac{1}{\beta}}}e^{-Ka^{\beta}}.
\end{align*}
If $c_n=e^{-K n^{\beta}}$, then
\begin{align*}
\dfrac{\sum_{n=N}^{\infty} (c_n)^2}{(\sum_{n=N}^{\infty} c_n)^2} & \geq \dfrac{(c_N)^2}{(\sum_{n=N}^{\infty} c_n)^2} \\
&\geq \dfrac{e^{-2K N^{\beta}}}{(e^{-K N^{\beta}}+\int_N^{\infty} e^{-K x^{\beta}}\,dx)^2} \geq \dfrac{1}{(1+ \beta^{-1} K^{-\frac{1}{\beta}})^2} >0
\end{align*}
for sufficiently large $N$, which shows that \eqref{eq:L2LLNcond} in Theorem \ref{thm:OT19MainResult} is not satisfied.
\end{example}

Example \ref{exs:ExpDame} shows that Theorem \ref{thm:OT19MainResult} does not cover the Takagi function $T(x)$ itself. In fact, we have the following result when $c_n=r^n$ with $0<|r|<1$.

\begin{theorem} \label{thm:OT19GansoIncluded} Let $r$ be a real number satisfying $0<|r|<1$. Define
\[ f_r(x) := \sum_{n=1}^{\infty} r^n \varphi^{(n)}(x) \]
and
\[ f_{r,N}(x) := \sum_{n=1}^{N-1} r^n \varphi^{(n)}(x) \quad \mbox{for $N=1,2,\cdots$.} \]
\begin{itemize}
\item[(i)] There is a nondegenerate random variable (i.e. a nonconstant measurable function) $L_r(x)$ with mean one such that
\[ \dfrac{f_r(x)-f_{r,N}(x)}{\tfrac{1}{2} \sum_{n=N}^{\infty} r^n} \to L_r(x) \quad \mbox{in distribution as $N \to \infty$}. \]
\item[(ii)] $\displaystyle \lim_{N' \to \infty} \dfrac{1}{N'} \sum_{N=1}^{N'} \dfrac{f_r(x)-f_{r,N}(x)}{\tfrac{1}{2}\sum_{n=N}^{\infty} r^n} = 1$ for a.e. $x$ and in $L^1$. 
\end{itemize}
\end{theorem}

\begin{remark} The limiting random variable $L_r(x)$ in Theorem \ref{thm:OT19GansoIncluded} (i) is nothing but $\dfrac{2(1-r)}{r} \cdot f_r(x)$.
\end{remark}

\begin{example} It is well-known (see e.g. \cite{YamagutiHata83HMJ}) that $f_{\frac{1}{4}}(x) = x(1-x)$. In this case
\[ L_{\frac{1}{4}} (x) = 6 f_{\frac{1}{4}}(x) =6x(1-x),\]
and for $0 \leq u \leq 6 \cdot (1/4)=3/2$, we have
\begin{align*}
&P(\{ x \in [0,1] :  L_{\frac{1}{4}} (x) \leq u\}) 
= 1- \sqrt{1- \dfrac{2u}{3}} = \int_0^u \dfrac{1}{\sqrt{9-6t}}\,dt.
\end{align*}

\end{example}

\section{Proof of Theorem \ref{thm:OT19MainResult}}

\subsection{Preliminaries}

We follow the probabilistic approach to Takagi class functions pioneered by K\^{o}no \cite{Kono87ActaMathHungar}. It is easy to see that the Lebesgue measure $P$ on $[0,1]$ is $\varphi$-invariant, which means that for each $n=1,2,\cdots$, the random variable $\varphi^{(n)}$ is uniformly distributed over $[0,1]$. Now we introduce 
\[ \varphi_*^{(n)}(x) := \varphi^{(n)}(x) - \dfrac{1}{2}\quad \mbox{for $n=1,2,\cdots$.} \]
Since $\varphi_*^{(n)}$ is uniformly distributed over $[-1/2,1/2]$, we can see that
\[ E[\varphi_*^{(n)}]=0,\quad 
E[(\varphi_*^{(n)})^2]=\dfrac{1}{12},\quad \mbox{and} \quad E[(\varphi_*^{(n)})^4]=\dfrac{1}{80}, \]
where $E[X]$ denotes the expectation of $X$ with respect to $P$: 
\[ E[X] := \int_0^1 X(x)\,dx. \]

The binary expansion of $x \in [0,1)$ is denoted by 
\[ \displaystyle x=\sum_{n=1}^{\infty} \dfrac{\varepsilon_n(x)}{2^n}.  \]
(If $x$ is a dyadic rational, then we choose the representation with $\varepsilon_n (x)=0$ except finitely many $n$'s.) We define {\it Rademacher functions} $\{R_n(x)\}$ by
\[ R_n(x) := 1-2\varepsilon_n(x)\quad \mbox{for $n=1,2,\cdots$.} \]

\begin{lemma}[\cite{Kono87ActaMathHungar}, Lemma 1] \label{lem:Kono87Lem1} For $n=1,2,\cdots$,
\[ \varphi_*^{(n)}(x) = -2^{n-1} R_n(x) \sum_{k=n+1}^{\infty}\dfrac{R_k(x)}{2^k}. \]
\end{lemma}

For $N=1,2,\cdots$, let
\[ \mathcal{T}_N := \sigma(\{ R_n : n =N,N+1,\cdots \}). \]
Then $\{\mathcal{T}_N\}$ is a decreasing family of sub $\sigma$-fields of $\mathcal{F}$. 
For each $N=1,2,\cdots$, we define 
\[ 
\displaystyle m_N := E[f-f_N]=\dfrac{1}{2}\sum_{n=N}^{\infty} c_n\]
and
\[ M_N(x):=\sum_{n=N}^{\infty} c_n\varphi_*^{(n)}(x)=f(x)-f_N(x)-m_N. \]
The following fact, which follows from Lemma \ref{lem:Kono87Lem1}, is the starting point of our calculation. (Apparently it plays no major role in \cite{Kono87ActaMathHungar}.)

\begin{lemma}[\cite{Kono87ActaMathHungar}, Theorem 1 (vi)] \label{lem:Kono87Thm1vi} $(M_N,\mathcal{T}_N)$ is a reverse martingale.
\end{lemma}


For $n=1,2,\cdots$, let $d_n:=M_n - M_{n+1}=c_n\varphi_*^{(n)}$.
Since $E[d_n \mid\mathcal{T}_{n+1}]=0$, we have 
\begin{align*}
 E[ d_{n_1} d_{n_2} \cdots  d_{n_r}] &= E[ E[d_{n_1}\mid \mathcal{T}_{n_1+1}]d_{n_2} \cdots  d_{n_r}]=0
\end{align*}
for any integer $r>1$ and $n_1<n_2<\cdots<n_r$. In fact, 
$\{\varphi_*^{(n)}\}$ forms a multiplicative system (\cite{Kono87ActaMathHungar}, Theorem 1 (i)): For any integer $r>1$ and $n_1<n_2<\cdots<n_r$,
\[ E[ \varphi_*^{(n_1)} \varphi_*^{(n_2)}\cdots  \varphi_*^{(n_r)}] = 0.\]

\subsection{Law of Large numbers}

By the orthogonality of reverse martingale differences $\{d_n\}$, we have
\begin{align*}
 s_N^2:=E[(M_N)^2] &= E\left[ \left(\sum_{n=N}^{\infty} d_n\right)^2 \right] \\
 &= \sum_{n=N}^{\infty} E[(d_n)^2]=
\sum_{n=N}^{\infty} (c_n)^2 E[ (\varphi_*^{(n)})^2 ] = \dfrac{1}{12} \sum_{n=N}^{\infty} (c_n)^2.
\end{align*}
Theorem \ref {thm:OT19MainResult} (i) follows from
\[ E\left[ \left(\dfrac{f-f_N}{m_N} -1 \right)^2\right] = E\left[\left(\dfrac{M_N}{m_N}\right)^2\right] = \dfrac{E[(M_N)^2]}{(m_N)^2} = \dfrac{ \tfrac{1}{12}\sum_{n=N}^{\infty} (c_n)^2}{\left(\tfrac{1}{2}\sum_{n=N}^{\infty} c_n\right)^2}. \]

To prove Theorem \ref {thm:OT19MainResult} (ii), we use a tail sum analog of Lemma 1 in Azuma \cite{Azuma67}, whose proof is quite similar to the original one and is omitted.
 
\begin{lemma} \label{lem:Azuma67Lem1Tail} For any $\lambda \in \mathbb{R}$, 
\[ E[e^{\lambda |M_N|}] \leq 2 \exp\left\{\dfrac{\lambda^2}{2} \sum_{n=N}^{\infty} (c_n)^2\right\} . \]
\end{lemma}

By Markov's inequality and Lemma \ref{lem:Azuma67Lem1Tail} with $\lambda=c\left\{ \sum_{n=N}^{\infty} (c_n)^2\right\}^{-1}$, we obtain
\[ P(|M_N| > c) \leq 2 \exp\left[-\dfrac{c^2}{2} \left\{ \sum_{n=N}^{\infty} (c_n)^2\right\}^{-1}\right] \]
for any $c>0$. Thus we have
\begin{align*}
P\left(\left|\dfrac{M_N}{m_N}\right| > \varepsilon \right) &\leq 2 \exp\left\{-\dfrac{\varepsilon^2}{8} \cdot \dfrac{\left(\sum_{n=N}^{\infty} c_n \right)^2}{\sum_{n=N}^{\infty} (c_n)^2}\right\}
\end{align*}
for any $\varepsilon>0$.
By the assumption of Theorem \ref{thm:OT19MainResult} (ii), we have 
\[
\displaystyle \sum_{N=1}^{\infty} P\left(\left|\dfrac{M_N}{m_N}\right| > \varepsilon \right)<+\infty.
\]
A standard application of the Borel-Cantelli lemma (see e.g. Theorem 2.1.1 in \cite{Stout74}) gives Theorem \ref {thm:OT19MainResult} (ii).

\subsection{The central limit theorem}

We try to apply the following result, which is a special case of Corollary 3.4 in Hall and Heyde \cite{HallHeyde80}.

\begin{lemma}[cf. \cite{HallHeyde80}, Corollary 3.4] Suppose that $(M_N,\mathcal{T}_N)$ is a square-integrable reverse martingale. 
Let 
\[ s_N^2:=E[(M_N)^2] \quad \mbox{and}\quad Q_N^2:=\sum_{n=N}^{\infty} (d_n)^2, \]
where $d_n:=M_n-M_{n+1}$. If both
\begin{align}
 \lim_{N \to \infty} \dfrac{1}{s_N^2} E\left[ \sup_{n \geq N} (d_n)^2 \right] =0
 \label{eq:HallHeyde80Cor3.4-1}
\end{align}
and 
\begin{align}
 \lim_{N \to \infty} \dfrac{Q_N^2}{s_N^2} = 1 \mbox{ in probability}
 \label{eq:HallHeyde80Cor3.4-2}
\end{align}
hold, then 
\[ \lim_{N \to \infty} \dfrac{M_N}{\sqrt{s_N^2}} =  \lim_{N \to \infty} \dfrac{M_N}{\sqrt{Q_N^2}}=N(0,1)  \mbox{ in distribution}. \]
\end{lemma}

Recalling that $d_n=c_n \varphi_*^{(n)}$, we have
\[ \sup_{n \geq N} (d_n)^2 = \sup_{n \geq N} (c_n \varphi_*^{(n)})^2 \leq \dfrac{1}{4} \sup_{n \geq N} (c_n)^2, \]
which says that \eqref{eq:HallHeyde80Cor3.4-1} follows from \eqref{eq:HallHeyde80Cor3.4-1suff1} in the lemma below.

\begin{lemma} The condition \eqref{eq:HallHeyde80Cor3.4-1suff1forULIL}
implies that
\begin{equation}
\lim_{N \to \infty} \dfrac{\sup_{n \geq N} (c_n)^2}{\sum_{n=N}^{\infty} (c_n)^2} = 0. \label{eq:HallHeyde80Cor3.4-1suff1}
\end{equation}
\end{lemma}

\begin{proof} Set 
\[ \gamma_N := \dfrac{(c_N)^2}{ \sum_{n=N}^{\infty} (c_n)^2} \]
and assume that $\displaystyle \lim_{N \to \infty} \gamma_N=0$.
Since 
\begin{align*}
 (c_m)^2 = \gamma_m \sum_{n=m}^{\infty} (c_n)^2 
 \leq \gamma_m \sum_{n=N}^{\infty} (c_n)^2
\end{align*}
for any $m \geq N$, we have 
\begin{align*}
 0 \leq \dfrac{\sup_{m \geq N} (c_m)^2}{\sum_{n=N}^{\infty} (c_n)^2} \leq \sup_{m \geq N} \gamma_m.
\end{align*}
Noting that $\displaystyle \lim_{N \to \infty} \sup_{m \geq N} \gamma_m=\lim_{N \to \infty} \gamma_N=0$, we have \eqref{eq:HallHeyde80Cor3.4-1suff1}. 
\end{proof}

To show \eqref{eq:HallHeyde80Cor3.4-2}, we prove
\begin{align}
\lim_{N \to \infty} E\left[ \left( \dfrac{Q_N^2}{s_N^2} -1 \right)^2 \right]=0. \label{eq:HallHeyde80Cor3.4-1suff2bis}
\end{align}
Note that
\begin{align*}
E\left[ \left( \dfrac{Q_N^2}{s_N^2} -1 \right)^2 \right] &= \dfrac{E\left[ (Q_N^2 - E[Q_N^2] )^2 \right]}{s_N^4} 
= \dfrac{E[(Q_N^2)^2] - (E[Q_N^2])^2}{s_N^4}.
\end{align*}
The numerator is equal to
\begin{align}
&E\left[ \left\{ \sum_{n=N}^{\infty} (c_n \varphi_*^{(n)})^2 \right\}^2 \right] - \left\{ E\left[\sum_{n=N}^{\infty} (c_n\varphi_*^{(n)})^2 \right] \right\}^2 \notag \\
&= \sum_{n=N}^{\infty} (c_n)^4 \left\{ E[(\varphi_*^{(n)})^4]-E[(\varphi_*^{(n)})^2]^2 \right\}  \notag \\
&\quad +2\sum_{N \leq i <j} (c_i)^2 (c_j)^2 \left\{ E[(\varphi_*^{(i)})^2(\varphi_*^{(j)})^2]-E[(\varphi_*^{(i)})^2]E[(\varphi_*^{(j)})^2] \right\}. \label{eq:OsakaTakeiCLTL2}
\end{align}
The first term in the right hand side is
\begin{align*}
\left(\dfrac{1}{80}-\dfrac{1}{144}\right) \sum_{n=N}^{\infty} (c_n)^4 = \dfrac{1}{180}\sum_{n=N}^{\infty} (c_n)^4.
\end{align*}
To estimate the second term in the right hand side, we prepare a lemma.

\begin{lemma} \label{lem:OsakaMixing} For any positive integers $i,j$ with $i<j$,
\[ E[(\varphi_*^{(i)})^2(\varphi_*^{(j)})^2]-E[(\varphi_*^{(i)})^2]E[(\varphi_*^{(j)})^2] = \dfrac{1}{180} \cdot \dfrac{1}{4^{j-i}}. \]
\end{lemma}

\begin{proof} Since $(R_i,R_{i+1},\cdots)$ has the same distribution as $(R_1,R_2,\cdots)$, 
\[ E[(\varphi_*^{(i)})^2(\varphi_*^{(j)})^2] = E[(\varphi_*^{(1)})^2(\varphi_*^{(j-i+1)})^2] \]
by Lemma \ref{lem:Kono87Lem1}. Let $n=j-i$. Again by Lemma \ref{lem:Kono87Lem1},
\begin{align*}
&E[(\varphi_*^{(1)})^2(\varphi_*^{(n+1)})^2] = (2^n)^2 \cdot E\left[ \left( \sum_{k=2}^{\infty} \dfrac{R_k}{2^k} \right)^2 \cdot  \left( \sum_{k'=n+2}^{\infty} \dfrac{R_{k'}}{2^{k'}} \right)^2\right].
\end{align*}
Using the independence of $\{R_k\}$,
\begin{align*}
E\left[ \left( \sum_{k=2}^{\infty} \dfrac{R_k}{2^k} \right)^2 \cdot  \left( \sum_{k'=n+2}^{\infty} \dfrac{R_{k'}}{2^{k'}} \right)^2\right]
&=E\left[ \left( \sum_{k=2}^{n+1}\dfrac{R_k}{2^k} \right)^2 \right] \cdot E\left[ \left( \sum_{k'=n+2}^{\infty} \dfrac{R_{k'}}{2^{k'}} \right)^2\right] \\
&\quad + 2 E\left[ \sum_{k=2}^{n+1}\dfrac{R_k}{2^k}  \right] \cdot E\left[ \left( \sum_{k'=n+2}^{\infty} \dfrac{R_{k'}}{2^{k'}} \right)^3\right] \\
&\quad + E\left[ \left( \sum_{k'=n+2}^{\infty} \dfrac{R_{k'}}{2^{k'}} \right)^4\right].
\end{align*}
Let us look at the right hand side. The first term is
\begin{align*}
\left\{ \sum_{k=2}^{n+1} \dfrac{1}{(2^k)^2} \right\} \cdot \dfrac{E[(\varphi_*^{(n+1)})^2]}{(2^n)^2} &=\dfrac{1}{4^2} \cdot \frac{1-\tfrac{1}{4^n}}{1-\tfrac{1}{4}} \cdot \dfrac{E[(\varphi_*^{(n+1)})^2]}{4^n} \\
&=\dfrac{1}{4^n}\left( 1-\frac{1}{4^n}\right) E[(\varphi_*^{(1)})^2] E[(\varphi_*^{(n+1)})^2].
\end{align*}
The second term is $0$. The third term is
\begin{align*}
\dfrac{E[(\varphi_*^{(n+1)})^4]}{(2^n)^4} = \dfrac{1}{80} \cdot \dfrac{1}{4^{2n}}.
\end{align*}
Thus we have
\begin{align*}
E[(\varphi_*^{(1)})^2(\varphi_*^{(n+1)})^2]-E[(\varphi_*^{(1)})^2]E[(\varphi_*^{(n+1)})^2] &= \left( \dfrac{1}{80}-\dfrac{1}{144} \right)\cdot \dfrac{1}{4^n} \\
&= \dfrac{1}{180} \cdot \dfrac{1}{4^n}.
\end{align*}
\end{proof}

Using Lemma \ref{lem:OsakaMixing}, the second term of the right hand side of \eqref{eq:OsakaTakeiCLTL2} is
\begin{align*}
2\sum_{i=N}^{\infty} (c_i)^2 \sum_{j=i+1}^{\infty} (c_j)^2 \cdot \dfrac{1}{180} \cdot \dfrac{1}{4^{j-i}} 
&\leq \dfrac{1}{90} \left(\sup_{j \geq N} (c_j)^2\right) \left(\sum_{i=N}^{\infty} (c_i)^2 \sum_{j=i+1}^{\infty}  \dfrac{1}{4^{j-i}}\right) \\
&= \dfrac{1}{270} \left(\sup_{j \geq N} (c_j)^2\right) \left(\sum_{i=N}^{\infty} (c_i)^2\right).
\end{align*}
Thus we have
\begin{align*}
E\left[ \left( \dfrac{Q_N^2}{s_N^2} -1 \right)^2 \right] \leq \dfrac{\frac{1}{270} \left(\sup_{j \geq N} (c_j)^2\right) \left(\sum_{i=N}^{\infty} (c_i)^2\right)}{\left(\frac{1}{12} \sum_{n=N}^{\infty} (c_n)^2 \right)^2} 
=\dfrac{8}{15} \cdot \dfrac{\sup_{j \geq N} (c_j)^2}{\sum_{n=N}^{\infty} (c_n)^2},
\end{align*}
which says that \eqref{eq:HallHeyde80Cor3.4-1suff1} implies \eqref{eq:HallHeyde80Cor3.4-1suff2bis}. This completes the proof of Theorem \ref{thm:OT19MainResult} (iii).

\subsection{The law of iterated logarithm}

The first half of Theorem \ref{thm:OT19MainResult} (iv) can be obtained by a similar idea as in the proof of Theorem 2 of Azuma \cite{Azuma67} (see also p. 238--240 in Stout \cite{Stout74}). We use the following lemma.
\begin{lemma}[cf. \cite{Azuma67}, Lemma 2] \label{lem:Azuma67Lem2Tail}
Let $a$ and $b$ be positive integers with $a<b$. For any real number $\lambda$,
\[ E\left[ \exp\left( \lambda \max_{a \leq N \leq b} \left| \sum_{n=N}^{\infty} c_n \varphi_*^{(n)} \right| \right) \right] \leq 8\exp\left\{ \dfrac{\lambda^2}{2} \sum_{n=a}^{\infty} (c_n)^2 \right\}. \]
\end{lemma}
\begin{proof} For a reversed martingale $(M_N,\mathcal{T}_N)$, the sequence $(Y_n, \mathcal{G}_n)_{n=0,\cdots,b-a}$ defined by
\[ Y_n := M_{b-n}, \quad \mbox{and} \quad \mathcal{G}_n := \mathcal{T}_{b-n} \]
forms a martingale. Using this observation we can prove Lemma \ref{lem:Azuma67Lem2Tail} by a similar argument to Lemma 2 in \cite{Azuma67}. 
\end{proof}

Recall that 
\[ s_N^2=\dfrac{1}{12} \sum_{n=N}^{\infty} (c_n)^2 \to 0 \quad \mbox{as $N \to \infty$.}
\]
We will show that
\begin{equation} \label{eq:Azuma67Theorem2Concl}
P\left(\mbox{$\displaystyle \left| \sum_{n=N}^{\infty} c_n \varphi_*^{(n)} \right| > (1+\varepsilon) \phi(s_N^2)$ for infinitely many $N$}\right) 
=0.
\end{equation}
Fix $\varepsilon>0$, and $p>1$ satisfying $(1+\varepsilon)/p>1$. We define a subsequence $\{N_k\}$ by 
\begin{equation}
 N_k := \min\left\{ N : \dfrac{1}{s_N^2} > p^k \right\}.\label{eq:Stout74p240subseq}
\end{equation}
Noting that \eqref{eq:HallHeyde80Cor3.4-1suff1forULIL} is equivalent to
\[ \lim_{N \to \infty} \dfrac{s_{N+1}^2}{s_N^2} = 1, \]
we have
\begin{equation}
\dfrac{1}{s_{N_k}^2} \sim p^k \quad (k \to \infty),
\quad \mbox{and} \quad
\lim_{k \to \infty} \dfrac{s_{N_{k+1}}^2}{s_{N_k}^2} = \dfrac{1}{p}.
\label{eq:Azuma67Theorem2subseq}
\end{equation}
Since $\phi(t)$ is increasing near $t=0$, we can see that
\begin{align*}
&P\left( \bigcup_{N=N_k}^{N_{k+1}-1} \left\{ \left| \sum_{n=N}^{\infty} c_n \varphi_*^{(n)} \right| > (1+\varepsilon) \phi(s_N^2) \right\} \right) \\
&\leq P\left( \bigcup_{N=N_k}^{N_{k+1}-1} \left\{ \left| \sum_{n=N}^{\infty} c_n \varphi_*^{(n)} \right| > (1+\varepsilon) \phi(s_{N_{k+1}}^2) \right\} \right)=:(*).
\end{align*}
By Lemma \ref{lem:Azuma67Lem2Tail} and Markov's inequality,
\begin{align*}
P\left( \max_{N_k \leq N < N_{k+1}} \left| \sum_{n=N}^{\infty} c_n \varphi_*^{(n)} \right| > u \right) \leq 8\exp\left( \dfrac{\lambda^2s_{N_k}^2}{2}   - \lambda u \right)
\end{align*}
for any $u>0$. Applying this to
\begin{align*}
\lambda=\dfrac{(1+\varepsilon) \phi(s_{N_{k+1}}^2)}{s_{N_k}^2} 
\quad \mbox{and} \quad
u=(1+\varepsilon) \phi(s_{N_{k+1}}^2),
\end{align*}
we have
\begin{align*}
(*)&\leq 8\exp\left\{ -\dfrac{(1+\varepsilon)^2\phi(s_{N_{k+1}}^2)^2}{2s_{N_k}^2}  \right\} \\
&=8\exp\left\{ -(1+\varepsilon)^2 \cdot \dfrac{s_{N_{k+1}}^2}{s_{N_k}^2} \cdot \log\log \left(\dfrac{1}{s_{N_{k+1}}^2}\right) \right\}.
\end{align*}
By \eqref{eq:Stout74p240subseq} and \eqref{eq:Azuma67Theorem2subseq},
\begin{align*}
\dfrac{1}{s_{N_{k+1}}^2} \geq p^{k+1} > p^k,
\mbox{ and } 
\dfrac{s_{N_{k+1}}^2}{s_{N_k}^2} \geq \dfrac{1+\varepsilon/2}{1+\varepsilon} \cdot \dfrac{1}{p} \mbox{ for sufficiently large $k$.}
\end{align*}
Thus we have
\begin{align*}
(*)&\leq 8\exp\left\{ -\dfrac{(1+\varepsilon)(1+\varepsilon/2)}{p} \log\log  p^k \right\},
\end{align*}
and we can find a positive constant $C$ satisfying $(*) \leq Ck^{-(1+\varepsilon/2)}$.
The Borel-Cantelli lemma implies \eqref{eq:Azuma67Theorem2Concl}.

To prove the second half of Theorem \ref{thm:OT19MainResult} (iv), we use the following lemma, which is a special case of Scott and Huggins \cite{ScottHuggins83}, Theorem 6.
\begin{lemma}[cf. \cite{ScottHuggins83}, Theorem 6] \label{thm:ScottHuggins83Thm6s} Assume that $(M_N,\mathcal{T}_N)$ is a reversed martingale with uniformly bounded differences. If a positive non-increasing sequence $\{W_N\}$ satisfies that
\begin{align}
 &\lim_{N \to \infty} \dfrac{1}{W_N^2} \sum_{n=N}^{\infty} E[(d_n)^2\mid \mathcal{T}_{n+1}] = 1\quad \mbox{a.s.,} \label{eq:ScottHuggins83(3.7)s}\\
 & \sum_{N=1}^{\infty}\dfrac{1}{W_N^4} E[(d_N)^4 \mid \mathcal{T}_{N+1}] <+\infty \quad \mbox{a.s.,} \label{eq:ScottHuggins83(3.8)s}\\
 &\lim_{N \to \infty} \dfrac{W_N}{W_{N+1}}=1, \quad \mbox{and}\label{eq:ScottHuggins83(3.9a)s} \\
 &\lim_{N \to \infty} W_N=0, \label{eq:ScottHuggins83(3.9b)s}
\end{align}
then 
\[ \limsup_{N \to \infty} \pm \dfrac{M_N}{\phi(W_N^2)} =1 \quad \mbox{a.s..}\]
\end{lemma}

\begin{remark} We explain the difference between the above lemma and Theorem 6 of \cite{ScottHuggins83}. Here we only use a deterministic norming sequence $\{W_N\}$. As $\{M_N\}$ has uniformly bounded differences, the sequence $\{Z_n\}$ in \cite{ScottHuggins83}, used for truncation, is not needed. Finally, the conclusion of Theorem 6 of \cite{ScottHuggins83} is the functional law of the iterated logarithm, while our statement is in an ordinary form for simplicity.
\end{remark}

We apply this lemma to $W_N^2 := s_N^2$.
As in the first half, \eqref{eq:ScottHuggins83(3.9a)s} and \eqref{eq:ScottHuggins83(3.9b)s} hold. Using \eqref{eq:TakagiClassLIL190310-2} and
\begin{align*}
\dfrac{1}{s_N^4} E[(d_N)^4 \mid \mathcal{T}_{N+1}] \leq \dfrac{\frac{1}{16}(c_N)^4 }{\left\{ \frac{1}{12}\sum_{n=N}^{\infty} (c_n)^2 \right\}^2}, 
\end{align*}
we have \eqref{eq:ScottHuggins83(3.8)s}. By Lemma \ref{lem:Kono87Lem1},
\[ (\varphi_*^{(n)})^2 = \left(-2^{n-1} R_n \sum_{k=n+1}^{\infty}\dfrac{R_k}{2^k}\right)^2 = \left(-2^{n-1} \sum_{k=n+1}^{\infty}\dfrac{R_k}{2^k}\right)^2 \]
is $\mathcal{T}_{n+1}$-measurable, and we have
\begin{align*}
\dfrac{1}{s_N^2} \sum_{n=N}^{\infty} E[(d_n)^2\mid \mathcal{T}_{n+1}] &= \dfrac{1}{s_N^2} \sum_{n=N}^{\infty} (d_n)^2 =\dfrac{Q_N^2}{s_N^2}.
\end{align*}
This means that \eqref{eq:ScottHuggins83(3.7)s} is the almost-sure version of 
\eqref{eq:HallHeyde80Cor3.4-2}. Let $\mathcal{F}_0$ denote the trivial $\sigma$-algebra, and $\mathcal{F}_n:=\sigma(R_1,\cdots,R_n)$ for $n=1,2,\cdots$. Noting that $\varphi_*^{(n)}$ is independent of $\mathcal{F}_{n-1}$ again by Lemma \ref{lem:Kono87Lem1}, 
\begin{align*}
\sum_{n=N}^{\infty} E[(d_n)^2\mid \mathcal{F}_{n-1}] = \sum_{n=N}^{\infty} E[(d_n)^2] = \sum_{n=N}^{\infty} (c_n)^2 E[(\varphi_*^{(n)})^2] = s_N^2,
\end{align*}
which implies
\begin{align} \dfrac{Q_N^2}{s_N^2}-1 = \dfrac{1}{s_N^2} \sum_{n=N}^{\infty}  \{ (d_n)^2 - E[(d_n)^2\mid \mathcal{F}_{n-1}] \}. \label{eq:OsakaTakei19LILfromsLLN}
\end{align}
By \eqref{eq:TakagiClassLIL190310-2},
\begin{align*}
\sum_{N=1}^{\infty} \dfrac{1}{s_N^4} E[(d_N)^4 \mid \mathcal{F}_{N-1}] <+\infty,
\end{align*}
and Theorem 2.15 in \cite{HallHeyde80} implies that
\begin{align*}
\sum_{N=1}^{\infty} \dfrac{1}{s_N^2} \{ (d_N)^2 - E[(d_N)^2\mid \mathcal{F}_{N-1}] \} <+\infty \quad \mbox{a.s..}
\end{align*}
By a tail version of Kronecker's lemma (see Lemma 1 (ii) in \cite{Heyde77}), the right hand side of \eqref{eq:OsakaTakei19LILfromsLLN} converges to $0$ as $N \to \infty$. This completes the proof.

\section{Proof of Theorem \ref{thm:OT19GansoIncluded}} 

Our proof of Theorem \ref{thm:OT19GansoIncluded} relies on the following identity.

\begin{lemma} \label{lem:rTakagiSelfSimilar} For any $N=1,2,\cdots$,
\[ \dfrac{f_r(x) - f_{r,N}(x)}{\frac{1}{2}\sum_{n=N}^{\infty} r^n} = \dfrac{2(1-r)}{r}\cdot f_r(2^{N-1} x)= \dfrac{f_r(2^{N-1} x)}{E[f_r]}. \]
\end{lemma} 
\begin{proof}
Recalling that $\varphi^{(n)} (x) = \varphi(2^{n-1}x)$, we see 
\[ \varphi^{(n+k)} (x) = \varphi^{(n)} (2^k x)\quad \mbox{for $n,k=1,2,\cdots$.} \]
Using this, we have
\begin{align*}
f_r(x) - f_{r,N}(x) &= \sum_{n=N}^{\infty} r^n \varphi^{(n)}(x) \\
&= \sum_{n=1}^{\infty} r^{n+N-1} \varphi^{(n+N-1)}(x) \\
&= r^{N-1} \sum_{n=1}^{\infty} r^n \varphi^{(n)}(2^{N-1}x) \\
&= r^{N-1} f_r(2^{N-1} x) 
= \dfrac{1-r}{r} \left(\sum_{n=N}^{\infty} r^n\right) f_r(2^{N-1} x).
\end{align*}
Since 
\[ E[f_r] = \sum_{n=1}^{\infty} r^n E[\varphi^{(n)}] = \dfrac{1}{2} \cdot \dfrac{r}{1-r}, \]
we obtain the desired identity.
\end{proof}

\begin{remark} The formula for $r=1/2$ is 
\[ f_{\frac{1}{2}}(x) - f_{\frac{1}{2},N}(x) = \dfrac{1}{2^{N-1}} f_r(2^{N-1} x), \]
which shows the self-similarity of the graph of the Takagi function $T(x)=f_{\frac{1}{2}}(x)$.
\end{remark}

The dyadic transformation of $\Omega$, defined by $B(x)=2x -\lfloor 2x \rfloor$, is one of the fundamental examples in ergodic theory (see e.g. Chapter 1 of Billingsley \cite{Billingsley65}). Since $B$ is measure-preserving, for any $N$, the distribution of $\dfrac{f_r(2^{N-1} x)}{E[f_r]}$ is the same as that of $\dfrac{f_r(x)}{E[f_r]}$. This together with Lemma \ref{lem:rTakagiSelfSimilar} gives Theorem \ref{thm:OT19GansoIncluded} (i). Moreover, since $B$ is an ergodic transformation of $\Omega$, it follows from Birkhoff's individual ergodic theorem and von Neumann's mean ergodic theorem that 
\begin{align*}
\lim_{N' \to \infty} \dfrac{1}{N'} \sum_{N=1}^{N'} \dfrac{f_r(2^{N-1} x)}{E[f_r]}
&= \int_0^1 \dfrac{f_r(x)}{E[f_r]} \,dx =1
\quad \mbox{a.s. and in $L^1$.}
\end{align*}
In view of Lemma \ref{lem:rTakagiSelfSimilar}, this completes the proof of Theorem \ref{thm:OT19GansoIncluded} (ii).

\section{Conclusion}
In this article, precise results on the rate of convergence of Takagi class functions $f$ are obtained. It tells us how accurate an approximation of $f$ by its partial sums, and can be applied in drawing an accurate graph of $f$. It also reveals new aspects of chaotic behavior of the tail sum of Takagi class functions. The reader is invited to explore potential applications of the idea developed in the present paper.

\appendix
\section{Appendix} 

Once one recognizes that \eqref{ineq:subexpasymp} should hold, it can be quickly obtained by summing up the asymptotics
\[ N^{1-\beta} e^{-KN^{\beta}} - (N+1)^{1-\beta} e^{-K(N+1)^{\beta}} = (1+o(1))K\beta e^{-KN^{\beta}} \quad \mbox{as $N \to \infty$}, \]
which can be proved by differentiation of $N^{1-\beta} e^{-KN^{\beta}}$ in $N$.
For the sake of completeness, we give a constructive proof of \eqref{ineq:subexpasymp}.

\begin{lemma} \label{lem:cnsubexpstrict} For $K>0$ and $0<\beta < 1$, there exist positive constants $C_i=C_i(K,\beta)$ $(i=1,2)$ such that
\begin{align*}
\int_a^{\infty} e^{-K x^{\beta}}\,dx &\geq C_1 a^{1-\beta} e^{-Ka^{\beta}} \qquad \mbox{for any $a>0$, and} \\
\int_a^{\infty} e^{-K x^{\beta}}\,dx &\leq C_2 a^{1-\beta} e^{-Ka^{\beta}} \qquad \mbox{for $a\geq 1$ satisfying $K a^{\beta} \geq 1$.} 
\end{align*}
\end{lemma}

\begin{proof} 
By changing the variable to $t=Kx^{\beta}$, 
\begin{align*}
\int_a^{\infty} e^{-K x^{\beta}}\,dx &= \dfrac{1}{\beta K^{\frac{1}{\beta}}}\int_{Ka^{\beta}}^{\infty} t^{\frac{1}{\beta}-1} e^{-t}\,dt.
\end{align*}
Using integration by parts, we have
\begin{align*}
\int_{Ka^{\beta}}^{\infty} t^{\frac{1}{\beta}-1} e^{-t}\,dt 
&=\left[ -t^{\frac{1}{\beta}-1} e^{-t}\right]_{Ka^{\beta}}^{\infty} + \int_{Ka^{\beta}}^{\infty} \left(\frac{1}{\beta}-1\right) t^{\frac{1}{\beta}-2}e^{-t}\,dt \\
&\geq (Ka^{\beta})^{\frac{1}{\beta}-1} e^{-Ka^{\beta}} = K^{\frac{1}{\beta}-1} a^{1-\beta} e^{-Ka^{\beta}}.
\end{align*}
For the other direction, let $s:= \left\lceil \dfrac{1}{\beta}-1 \right\rceil$.
By repeated use of integration by parts, 
\begin{align*}
\int_{Ka^{\beta}}^{\infty} t^{\frac{1}{\beta}-1} e^{-t}\,dt 
&=(Ka^{\beta})^{\frac{1}{\beta}-1} e^{-Ka^{\beta}} + \dfrac{\Gamma(\tfrac{1}{\beta})}{\Gamma(\tfrac{1}{\beta}-1)}\int_{Ka^{\beta}}^{\infty} t^{\frac{1}{\beta}-2}e^{-t}\,dt \\
&= \cdots \\
&= \sum_{i=1}^s \dfrac{\Gamma(\tfrac{1}{\beta})}{\Gamma(\tfrac{1}{\beta}+i-1)} (Ka^{\beta})^{\frac{1}{\beta}-i}  e^{-Ka^{\beta}} + \dfrac{\Gamma(\tfrac{1}{\beta})}{\Gamma(\tfrac{1}{\beta}-s)}\int_{Ka^{\beta}}^{\infty} t^{\frac{1}{\beta}-s-1}e^{-t}\,dt.
\end{align*}
For $a \geq 1$, we have $a^{1-i\beta } \leq a^{1-\beta}$ for $i=1,\cdots,s$. 
Noting that $(1/\beta)-s-1 \leq 0$, for $a \geq 1$ satisfying $Ka^{\beta} \geq 1$, the second term in the right hand side is not more than
\[ \dfrac{\Gamma(\tfrac{1}{\beta})}{\Gamma(\tfrac{1}{\beta}-s)}\int_{Ka^{\beta}}^{\infty} e^{-t}\,dt = \dfrac{\Gamma(\tfrac{1}{\beta})}{\Gamma(\tfrac{1}{\beta}-s)} \cdot e^{-Ka^{\beta}}. \]
Thus we have
\[ \int_{Ka^{\beta}}^{\infty} t^{\frac{1}{\beta}-1} e^{-t}\,dt \leq \left( \sum_{i=1}^s \dfrac{\Gamma(\tfrac{1}{\beta})}{\Gamma(\tfrac{1}{\beta}+i-1)}K^{\frac{1}{\beta}-i} + \dfrac{\Gamma(\tfrac{1}{\beta})}{\Gamma(\tfrac{1}{\beta}-s)}\right)a^{1-\beta} e^{-Ka^{\beta}}. \]
\end{proof}

By this lemma, we have
\begin{align*}
\sum_{n=N}^{\infty} e^{-K n^{\beta}} &\geq \int_N^{\infty} e^{-K x^{\beta}}\,dx 
\geq C_1 N^{1-\beta} e^{-KN^{\beta}} \\
\intertext{for any $N$, and}
\sum_{n=N}^{\infty} e^{-K n^{\beta}} 
&\leq e^{-KN^{\beta}} +\int_N^{\infty} e^{-K x^{\beta}}\,dx 
\leq (1+C_2) N^{1-\beta} e^{-KN^{\beta}}
\end{align*}
for sufficiently large $N$.

\section*{Acknowledgements}
The authors deeply thank anonymous referees for their valuable comments. In particular, one of referees suggests simplification of some of our original arguments, which are adopted in the revised version. M.T. is partially supported by JSPS Grant-in-Aid for Young Scientists (B) No. 16K21039, and JSPS  Grant-in-Aid for Scientific Research (C) No. 19K03514.

\end{document}